\documentclass[pdftex,11pt]{amsart}
\usepackage{amsmath}
\usepackage{amsthm}
\usepackage{amssymb}
\usepackage{graphicx}
\usepackage{comment}
\usepackage{ytableau}
\usepackage{braket}
\usepackage{mathrsfs}
\usepackage{nccmath}
\usepackage{tikz}
\usepackage{blkarray}
\usepackage[margin=20truemm]{geometry}
\usepackage{here}
\usetikzlibrary{intersections,calc,arrows.meta}
\theoremstyle{plain}
\newtheorem{thm}{Theorem}[section]
\newtheorem{prop}[thm]{Proposition}
\newtheorem{lemma}[thm]{Lemma}
\theoremstyle{definition}
\newtheorem{exa}[thm]{Example}
\newtheorem{rem}[thm]{Remark}
\newtheorem{defi}[thm]{Definition}
\newtheorem*{nota*}{Notation}
\newtheorem*{thm*}{Theorem}
\newtheorem{case}{Case}

\newcommand{\MF}{\mathrm{matching field}}
\newcommand{\RR}{\mathbb{R}}
\newcommand{\ZZ}{\mathbb{Z}}

\newcommand{\BL}{\mathcal{B}_\ell}

\newcommand{\conv}{{\rm Conv}}
\newcommand{\MP}{{P}}
\def\wb{{\mathbf w}}

\newcommand{\Trop}[3][black]{\draw[#1,thick,domain=0:5-#2] plot(\x+#2,\x+#3);
 \draw[#1,thick,domain=-2-#2:0] plot(\x+#2,#3);
 \draw[#1,thick,domain=-2-#3:0] plot(#2,\x+#3);}
\newcommand{\wTrop}[3][black]{\draw[#1, thick,domain=0:7-#2] plot(\x+#2,\x+#3);
 \draw[#1, thick,domain=-2-#2:0] plot(\x+#2,#3);
 \draw[#1, thick,domain=-20-#3:0] plot(#2,\x+#3);}

\let\int\relax
\DeclareMathOperator{\int}{int}

\begin{document}
\title{Tropical hyperplane arrangements and combinatorial mutations of the matching field polytopes of Grassmannians}
\author{Nobukazu Kowaki}
\address[N. Kowaki]{Department of pure and applied mathematics, graduate school of information science and technology, Osaka University, Suita, Osaka 565-0871, Japan}
\email{u793177f@ecs.osaka-u.ac.jp}
\maketitle
\begin{abstract}
A sequence of combinatorial mutations of matching field polytopes preserves the property of giving rise to a toric degeneration of  Grassmannians. In this paper, we find a way to check that two matching field polytopes are combinatorial mutation equivalence using tropical hyperplane arrangements, ``literally at a glance". Our way can prove that block diagonal matching fields are combinatorial mutations equivalent to diagonal matching fields. This is one of main results in \cite{clarke2021combinatorial}. Our result can be regarded as a generalization of that result.
\end{abstract}
\section{Introduction}
A \textit{Grassmannian} $\mathrm{Gr}(k,n)$ is the set of all subspaces of a given dimension $k$ in $\mathbb{K}^n$, where $\mathbb{K}$ is a field. 

Toric degenerations provide the tool of polyhedral geometry to study algebraic geometry.
One of the motivations for toric degenerations of Grassmannians comes from mirror symmetry \cite{coates2022unwinding}. Therefore, we study toric degenerations of Grassmannians. In particular we focus on toric degeneration by using SAGBI basis.

SAGBI basis was independently introduced by Kapur and Madlener \cite{kapur1989completion} and Robbiano and Sweedler \cite{robbiano1990subalgebra}. SAGBI means Subalgebra Analogue of Gr\"{o}bner Bases for Ideals.  SAGBI basis answers the subalgebra membership problem similar to Gr\"{o}bner bases to the ideal membership problem. What is important for us is that toric degenerations of Grassmannians by matching fields are characterized by that Pl\"{u}cker coordinates form SAGBI basis.

We need the monomial ordering to use SAGBI basis. To give this, we use matching fields. Given integers $k$ and $n$, a \textit{matching field} denoted by $\Lambda(k,n)$, or $\Lambda$ when there is no confusion, is a choice of permutation $\Lambda(I) \in S_k$ for each $I \in \mathbf{I}_{k, n} = \{ I \subset [n] : |I| = k \}$, where $[n]=\{1,2,\ldots,n\}$ 
. 
A matching field has toric degenerations if Pl\"{u}cker coordinates form SAGBI basis for the subalgebra they generate, with respect to a monomial ordering the matching field induces. See Section \ref{sec:matching_fields} for more details.
 Matching fields were born to study the Newton polytope of the product of all maximal minors of a matrix of indeterminates $X=(x_{ij})$ by Sturmfels and Zelevinsky \cite{sturmfels1993maximal}. Matching field polytope is defined as the convex hull of the exponent vectors corresponding to the initial terms of maximal minors of a matrix of indeterminates $X=(x_{ij})$ by matching fields by Mohammadi and Shaw \cite{mohammadi2019toric}. The most famous example of matching fields is the diagonal matching field. This chooses the diagonal term as the initial for each minor. 
Namely, the diagonal matching field sends every $I \in I_{k,n}$ to $\mathrm{id} \in S_k$.

Mohammadi and Shaw introduced a more generalized class of matching fields, block diagonal matching fields \cite{mohammadi2019toric}. The image of this matching field is $\{ \mathrm{id},\ (1\ 2) \}$. Block diagonal matching fields were studied by many researchers \cite{clarke2021combinatorial,clarke2021toric,clarke2024toric,higashitani2022quadratic}. Note we still do not know the complete answer to which block diagonal matching field has toric degenerations. 

Checking if a coherent matching field gives toric degeneration is a difficult problem. However, in $\mathrm{Gr}(2,n)$, solving this problem is easy because all the coherent matching fields are the same as the diagonal matching field up to symmetry. 

How about $\mathrm{Gr}(3,n)$? In the case of $\mathrm{Gr}(3,n)$, it is proved in \cite[Theorem 1.3]{mohammadi2019toric} that a non-hexagonal $3\times n$ matching field whose ideal is quadratically generated provides a toric degeneration of $\mathrm{Gr}(3,n)$. At present, no non-hexagonal $3\times n$ matching fields whose ideal is not quadratically generated are known.

To study which matching field has toric degenerations for $\mathrm{Gr}(3,n)$, we use the combinatorial mutation. Combinatorial mutation was introduced by Akhtar, Coates, Galkin and Kasprzyk \cite{MR3007265}. Combinatorial mutation was applied by Clarke, Higashitani and Mohammadi in the context of Grassmannians \cite{clarke2021combinatorial}. 
Combinatorial mutation about matching field polytopes preserves whether the matching field gives toric degeneration \cite[Theorem 1]{clarke2021combinatorial}. 

However, checking combinatorial mutation equivalence is difficult because it is quite sensitive. So, we establish a new tool to solve the problems by using tropical hyperplane arrangements. This way can solve the problem, ``literally at a glance".

Our main theorem is as follows.
\setcounter{section}{3}
\setcounter{thm}{1}

\begin{thm}
Assume that there are two matching fields $\Lambda ,\ \Lambda '$, both of which have the same tropical hyperplane arrangements except for swapping two adjacent tropical lines, $i$ and $j$, and that condition $(\boldsymbol{\ast})$ is satisfied. Then, the pair of matching field polytopes $P_\Lambda ,\ P_{\Lambda '}$ can be obtained from one another by a combinatorial mutation.
\end{thm}

\setcounter{section}{1}
As a corollary of this theorem, we obtain there is a sequence of combinatorial mutation to the diagonal matching field for any block diagonal matching field. See Remark \ref{BDMFCor}.

We finish the introduction with an outline of the paper. 
Section $2.1$ fixes notations for the Grassmannians. In Section $2.2$, we introduce matching fields and the associated polytope. In Section $2.3$, we introduce the combinatorial mutation. Finally, our main theorem is proved in Section $3$ and an example of our main theorem is also given there. 
\section{Preliminaries}
\subsection{Grassmannians and Pl\"{u}cker coordinates}
A \textit{Grassmannian} $\mathrm{Gr}(k,n)$ is the set of all subspaces of a given dimension $k$ in $\mathbb{K}^n$, where $\mathbb{K}$ is a field.  For any subset $I = \{ i_1, \ldots ,i_k\} \subset [n]=\{1,2,\ldots,n\}$, and the $k\times n$ indeterminants matrix $X$, let $X_{I}$ be the submatrix with rows $1,\ldots ,k$ and columns $i_1,\ldots ,i_k$. Let $\mathbf{I}_{k, n} = \{ I \subset [n] : |I| = k \}$. The kernel of the map $$\varphi \colon\  \mathbb{K}[P_I:I\in \mathbf{I}_{k,n}]  \rightarrow \mathbb{K}[x_{ij}:1\leq i \leq k, 1 \leq j \leq n] \quad\text{with}\quad P_{I}   \mapsto \mathrm{det} (X_{I})$$
is called the \textit{Pl\"{u}cker ideal}. The \textit{Pl\"{u}cker coordinates} are the minors $\det (X_I)$, and the corresponding \textit{Pl\"{u}cker variable} is denoted by $P_I$. We identify Pl\"{u}cker coordinates with Pl\"{u}cker variables.
\subsection{Matching fields and their associated ideals}\label{sec:matching_fields}

We use many definitions from \cite{clarke2021combinatorial,mohammadi2019toric}.
Given integers $k$ and $n$, a \textit{matching field}, denoted by $\Lambda(k,n)$, or $\Lambda$ when there is no confusion, is a choice of permutation $\Lambda(I) \in S_k$ for each $I \in \mathbf{I}_{k, n} $
. We think of the permutation {$\sigma=\Lambda(I)$} as inducing a new ordering on the elements of $I$, where the position of $i_s$ is  $\sigma(s)$. In addition, we think of $I$ as being identified with a monomial of the Pl\"{u}cker variable $P_I$ and we represent these monomials as a $k \times 1$ tableau where the entry of $(\sigma(s), 1)$ is $i_{s}$. To make this tableau notation precise we define the ideal of the matching field as follows.

\medskip
Let $X=(x_{ij})$ be a $k \times n$ matrix of indeterminates. To every $I=\{i_1, \cdots, i_k\} \in \mathbf{I}_{k,n} \text{ with } i_1 <\cdots <i_k \text{ and } \sigma = \Lambda(I)$, let 
$
\textbf{x}_{\Lambda(I)}:=x_{1 i_{\sigma(1)}}x_{2 i_{\sigma(2)}}\cdots x_{{k i_{\sigma(k)}}}. 
$
The \textit{matching field ideal} $J_\Lambda$ is defined as the kernel of the monomial map
\begin{eqnarray}\label{eqn:monomialmap}
\varphi_{\Lambda} \colon\  & \mathbb{K}[P_I]  \rightarrow \mathbb{K}[x_{ij}]  
\quad\text{with}\quad
 P_{I}   \mapsto \text{sgn}(\Lambda(I)) \textbf{x}_{\Lambda(I)},
\end{eqnarray}
where $\text{sgn}(\Lambda(I))$ denotes the signature of the permutation $\Lambda(I)$ for each $I \in \mathbf{I}_{k, n}$. 
\medskip
\begin{defi}\label{def:matching_field}
A matching field $\Lambda$  is \textit{coherent} if there exists a $k\times n$ matrix $M=(m_{ij})$ 
with 
$m_{ij}\in\mathbb{R}$ 
such that 
for every  $I \in \mathbf{I}_{k,n}$ 
the initial of the Pl\"ucker coordinate  $\mathrm{det} (X_I) \in \mathbb{K}[x_{ij}]$ is $\text{in}_M (\mathrm{det}(X_I)) = \varphi_{\Lambda}(P_I)$, where $\text{in}_M (\mathrm{det}(X_I))$ is the sum of all terms in $\mathrm{det}(X_I)$ of the lowest weight and the weight of a monomial $x_{1i_1}\cdots x_{ki_k}$ is $m_{1i_1}+\cdots+m_{ki_k}$.
In this case, we say that the matrix $M$ \textit{induces the matching field} $\Lambda$. We let $\wb_M$ be the weight vector on the variables $P_I$ induced by the entries $m_{ij}$ of the weight matrix $M$ on the variables $x_{ij}$. More precisely, the weight of each variable $P_I$ is defined as the minimum weight of the terms of the corresponding minor of $M$, and it is called $\mathit{\text{\textit{the weight induced by }}M}$. 
\end{defi}
\medskip
\begin{exa}\label{ex:diag}
Consider the matching field $\Lambda(3,6)$ which assigns to each subset $I$ the identity permutation. Consider the following matrix:
\[
M=\begin{pmatrix}
     0  &0  & 0  & 0  & 0  & 0  \\
     6  &5  & 4 & 3  & 2  & 1\\
     11  &9 & 7 & 5 & 3  & 1 \\
\end{pmatrix}.
\]
The weights induced by $M$ on the variables $P_{123}, P_{124},\dotsc, P_{456}$ are $12, 10,\dotsc, 3$ respectively. Thus, for each $I=\{i,j,k\}$ we have that $\text{in}_M (\mathrm{det}(X_I)) = x_{1i}x_{2j}x_{3k}$ for $1\leq i<j<k\leq 6$. Therefore, the matrix $M$ induces $\Lambda(3,6)$. 

\medskip
Notice that each initial term $\text{in}_M(P_I)$ arises from the leading diagonal. Such matching fields are called \textit{diagonal}.
\end{exa}
As a generalization of the diagonal matching field, we define a \textit{block diagonal matching field}.
\begin{defi}[{\cite[Definition 4.1]{mohammadi2019toric}}]\label{def:block}
Given $k,n$ and $0\leq\ell\leq n$, we define the 
block diagonal matching field
$\BL$
as the map from $\mathbf{I}_{k,n}$ to $S_k$ such that
\[
 \BL(I)= \left\{
     \begin{array}{@{}l@{\thinspace}l}
      \mathrm{id} \; \; \; \; \; & \text{if $\lvert I|=1$ or $\lvert I \cap [\ell]\rvert \neq 1$},\\
      (12)  & \text{otherwise}. \\
     \end{array}
   \right.
\]
\end{defi}
\medskip

Given a matching field $\Lambda$, we associate to it a polytope $\MP_{\Lambda}$. The vertices of the polytope are in one-to-one correspondence with the tableaux of the matching field. In fact, reading the vertices of the polytope uniquely defines the matching field.
\begin{defi}[{\cite[Definition 6]{clarke2021combinatorial}}]\label{def:matching_field_polytope}
Fix $k$ and $n$. We take $\RR^{k \times n}$ to be the vector space of $k \times n$ matrices with canonical basis $\{ e_{i,j} : 1 \le i \le k, 1 \le j \le n\}$ where $e_{i,j}$ is the matrix with a $1$ in row $i$ and column $j$ and zeros everywhere else. Given a matching field $\Lambda$, for each $I=\{i_1,\ldots,i_k\} \in\mathbf{I}_{k,n}$ with $1 \le i_1< \cdots < i_k \le n$ we set $v_{I,\Lambda}:=\sum_{j=1}^ke_{j,i_{\Lambda(I)(j)}}$.
Then the \textit{matching field polytope} is
\[
\MP_{\Lambda} = \conv\left\{
v_{I,\Lambda}:\ I\in \mathbf{I}_{k,n}
\right\}.
\]
For notation we often write the tuple $(i_{\Lambda(I)(1)}, i_{\Lambda(I)(2)}, \dots, i_{\Lambda(I)(k)})$ for the vector $v_{I, \Lambda}$.
\end{defi}
\medskip
\subsection{Tropical maps and combinatorial mutations}
Let  $M=(m_{ij}) \in \mathbb{R}^{k \times n}$ be a weight matrix. For each $1 \leq j \leq n$ consider the  piecewise linear function $F_j: \mathbb{R}^{k-1} \to \mathbb{R}$, given by 
 \begin{eqnarray*}\label{eqn:linearform}
 F_j ((x_2,\cdots,x_k)) = \max \{ m_{1j}, m_{2j} + x_2, \dots, m_{kj}+ x_k\}. 
 \end{eqnarray*}

\medskip
A $k \times n$  weight matrix $M \in \mathbb{R}^{k \times n}$,  whose first row consists of zeros,  produces an arrangement of tropical hyperplanes $\mathcal{A} = \{ H_1, \dots , H_n\}$ defined by the functions  $F_1, \dots , F_n$, whose coefficients come from $M$. 
To each $k-1$ dimensional cell $\tau$ of the  complement of the arrangement $\mathcal{A}$ in  $\mathbb{R}^{k-1}$ there is an associated \emph{covector} $c_{\tau} \in \mathcal{P}[n]^k$, where $\mathcal{P}[n]$ denotes the power set of $n$. 
The $i$-th entry of the covector $c_{\tau}$ is a subset $S_i \subset [n]$ corresponding to the collection of  hyperplanes in $\mathcal{A}$ which intersect the ray  $x + tv_i$ for $x \in \tau^{\circ}$ and $t \geq 0$, where $\tau^{\circ}$ is the interior of $\tau$, and we let the vectors $v_i = -e_{i-1}$ for $i = 2, \dots, k-1$, and $v_1 = (1, \dots, 1)$. The \emph{coarse covector}
of a cell is simply the vector which records the sizes of the subsets of the covector. 
For example, the tropical hyperplane arrangement of the matrix inducing the diagonal matching field is as in Figure \ref{diagonal}.
\medskip

\begin{figure}[h]
\centering
\begin{tikzpicture}
\Trop{0}{0}
\Trop{1}{2}
\Trop{0.5}{1}
\Trop{1.5}{3}
\Trop{-0.5}{-1}
\Trop{2}{4}
\end{tikzpicture}
\caption{Tropicalization of the diagonal $\MF$ of Gr(3,6)}\label{diagonal}
\end{figure}
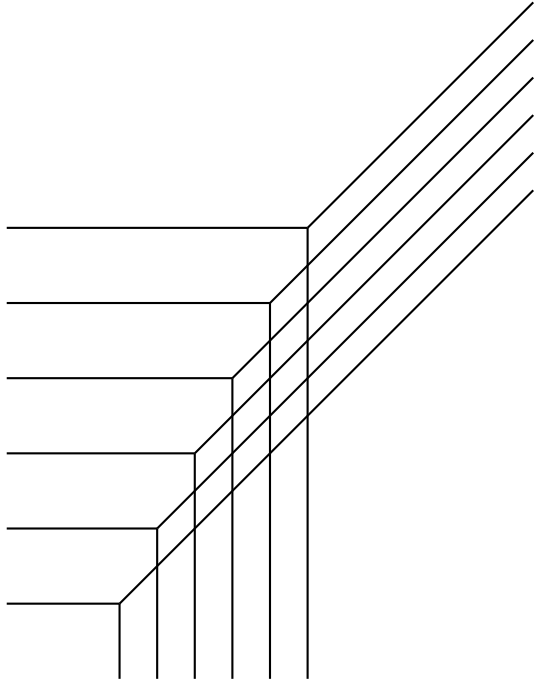
The following proposition can be extracted from \cite[Proposition 3.4]{mohammadi2019toric}. This proposition is a bridge between tropical geometry and initial degeneration.
\medskip
\begin{prop}[{\cite[Proposition 3.4]{mohammadi2019toric}}]\label{prop:initialPlucker}
Let $M$ be a $k \times n$ weight matrix, and let $\Lambda$ be the coherent matching field induced by $M$. Assume that for any $I = \{i_1, \ldots ,i_k\} \in \mathbf{I}_{k,n}$ with $i_1<\cdots <i_k$,   the collection of hyperplanes $\{H_i\}_{i\in I}$ intersects properly. 
Then
$$\mathrm{in}_{M} (\mathrm{det} (X_{I})) =\mathrm{sgn}({\Lambda(I)}) x_{1c_1}x_{2c_2}\dots x_{kc_{k}},$$
where $(\{c_1\}, \dots, \{c_{k}\})$ is the covector of the unique cell with coarse covector $(1, 1, \dots, 1)$ 
such that $i_{\Lambda(I)(j)}=c_j$ for $j=1,\ldots ,k$.
\end{prop}

\begin{figure}[h]
\centering
\begin{tikzpicture}[scale=0.6]
\Trop{0}{0}
\Trop{2}{4}
\Trop{1}{2}
\Trop[dashed]{1.5}{1.8}
\coordinate (A) at (0,-2);
 \draw (A) node[below]{$k$};
\coordinate (B) at (1,-2);
 \draw (B) node[below]{$i$};
\coordinate (C) at (2,-2);
 \draw (C) node[below]{$j$};
\end{tikzpicture}
\caption{This covector is $(\{j\},\{i\},\{k\})$. Therefore, we obtain the ordering $j,i,k$.}
\end{figure}
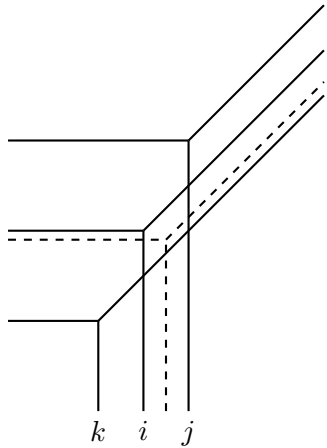
Tropical hyperplanes we deal with are the fan that has the three rays, i.e. $k=3$. Depending on the directions, we name every ray \textit{$x$-ray} $-e_1$, \textit{$y$-ray} $-e_2$, and \textit{diagonal ray} $e_1+e_2$. We call the intersection of those three rays the \textit{central point}.
The ray that has the inverse direction of a diagonal ray is the \textit{anti-diagonal ray}. Similarly, \textit{anti-$x$-ray} and \textit{anti-$y$-ray} are defined. See Figure \ref{nameofray}.

\medskip

\begin{figure}[h]
\centering
\begin{tikzpicture}
\Trop{1.5}{3}
\draw[dashed ,thick,domain=0:3.5] plot(-\x+1.5,-\x+3);
 \draw[dashed ,thick,domain=-3.5:0] plot(-\x+1.5,3);
 \draw[dashed ,thick,domain=-3.5:0] plot(1.5,-\x+3);
\filldraw (1.5, 3) circle (0.07);
\coordinate (A) at (2.6,3);
 \draw (A) node[below]{central point};
\coordinate (B) at (4.3,5);
 \draw (B) node[below]{diagonal ray};
\coordinate (C) at (2,-1);
 \draw (C) node[below]{$y$-ray};
\coordinate (D) at (-0.5,3);
 \draw (D) node[below]{$x$-ray};
\coordinate (E) at (-0.3,0);
 \draw (E) node[below]{anti-diagonal ray};
\coordinate (F) at (2.35,6);
 \draw (F) node[below]{anti-$y$-ray};
\coordinate (G) at (3.1,3.5);
 \draw (G) node[below]{anti-$x$-ray};
\end{tikzpicture}
\caption{The name of rays}\label{nameofray}
\end{figure}
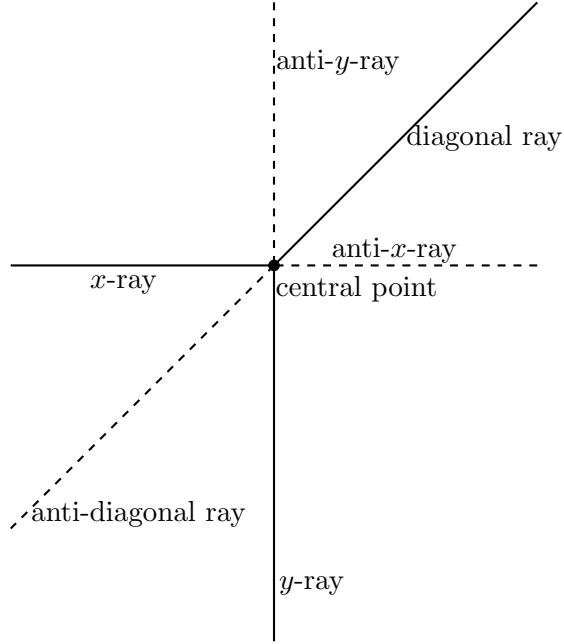

We begin by fixing two lattices $N = \ZZ^d$ and its dual $M = {\rm Hom}_\ZZ(N, \ZZ)$. We take $N_\RR = N \otimes_\ZZ \RR$ and similarly $M_\RR = M \otimes_\ZZ \RR$. We fix the standard inner product $\langle \cdot , \cdot \rangle : N_\RR \times M_\RR \rightarrow \RR$ given by evaluation $\langle v,u\rangle:=u(v)$ for $v\in N_\RR$ and $u\in M_\RR$. 
\medskip

\begin{defi} Let $w$ be a primitive lattice point of $M$ and $F \subset w^{\perp} \subset N_\RR$ a lattice polytope, where $w^{\perp}= \{v\in N_{\RR}| \langle v,w \rangle =0\}$.
The \textit{tropical map} defined by $w$ and $F$ is given by
\begin{equation*}
\varphi_{w,F}\ \colon\  M_\RR  \rightarrow M_\RR, \;\; u   \mapsto u-u_{\textrm{min}} w,
\end{equation*}
where $u_{\mathrm{min}}:=\mathrm{min}\{\langle u,p\rangle:p \in F\}$.\\
Let $\MP \subset M_\RR$ be a lattice polytope that contains the origin and suppose that $\varphi_{w,F}(\MP)$ is convex. Then we say that the polytope $\varphi_{w,F}(\MP)$ is a \textit{combinatorial mutation} of $\MP$. 
\end{defi}

\medskip
To check that the tropical map induces a combinatorial mutation, the following lemma is useful.
\medskip
\begin{lemma}[{\cite[Lemma 3]{clarke2021combinatorial}}]\label{lem:no_edge_in_image}
Let $\MP \subset M_{\RR}$ be a rational polytope and $f \in N$ a non-zero vector. Suppose that $\MP \subset \{x \in M_{\RR} : -1 \le \langle f,x \rangle \le 1 \}$. Write $\MP_+ = \{ x \in \MP : 0 \le \langle f,x \rangle \le 1 \}$ and $\MP_- = \{x \in \MP: -1 \le \langle f,x \rangle \le 0 \}$. Let $w \in M$ be a non-zero vector such that $f \in w^\perp$ and define $F = \conv\{0, f\}$. Let $\varphi = \varphi_{w,F}$ be the tropical map and note that $\varphi(\MP) = \varphi(\MP_+) \cup \varphi(\MP_-)$ is the union of two polytopes that intersect along a common face. 
Let $V(\varphi(\MP))$ be the vertices of $\varphi(\MP)$.
Then, there is no vertex of $\varphi(\MP)$ that is not the image of any vertex of $\MP$ if and only if for any vertices $u, v \in V(\MP)$ with $\langle f,u \rangle = -1 $, $\langle f,v \rangle = 1 $, there exist $t,t'\in P$ with $\langle f,t \rangle = \langle f,t'\rangle=0$ such that $t+t'=u+v$.
\end{lemma}
\medskip

\section{Main results}
Consider the arrangement of tropical hyperplanes associated with the coherent matching field of $\mathrm{Gr}(3,n)$. Let $M$ be a $3 \times n$ weight matrix and let $\{H_1,\cdots ,H_n\}$ be the tropical lines associated with $M$.
By abuse of notation, we identify $H_i$ with $i$. 
We say that two tropical lines $i$ and $j$ are \textit{adjacent} to each other if their $y$-rays are adjacent. Fix two adjacent tropical lines $i$ and $j$ and assume the $x$-coordinate of the central point of $j$ is larger than that of $i$. 

We define six colored regions. At first, assume the $y$-coordinate of the central point of $j$ is larger than that of $i$. 
\begin{itemize}
\item The red region is bounded by the anti-diagonal ray and $y$-ray of $i$.
\item The purple region is bounded by the anti-diagonal ray, anti-$y$-ray of $i$, and $x$-ray of $j$. 
\item The olive region is bounded by the anti-$y$ ray of $i$, and $x$-ray of $j$.
\item The blue region is the area bounded by the anti-$x$-ray and $y$-ray extending from the intersection of the diagonal ray of $i$ and the $y$-ray of $j$.
\item The green region is the area bounded by the anti-$x$-ray, anti-$y$-ray extending from the intersection of the diagonal ray of $i$ and the $y$-ray of $j$, and the diagonal ray of $j$.
\item The yellow region is bounded by the anti-$y$-ray and diagonal ray of $j$.
\end{itemize}
\begin{figure}[h]
\begin{tikzpicture}
 \fill[red!50](-2,-2)--(0,0)--(0,-2)--(-2,-2);
 \fill[blue!50](1,-2)--(1,1)--(3,1)--(3,-2);
 \fill[green!50](3,1)--(1,1)--(1,2)--(3,4);
 \fill[yellow!50](1,2)--(1,5)--(3,5)--(3,4);
 \fill[purple!50](-2,-2)--(0,0)--(0,2)--(-2,2);
 \fill[olive!50](-2,2)--(0,2)--(0,5)--(-2,5);
 \draw[thick,domain=0:3] plot(\x,\x);
 \draw[thick,domain=-2:0] plot(\x,0);
 \draw[thick,domain=-2:0] plot(0,\x);
 \draw[thick,dashed,domain=0:1.3] plot(\x,0);
 \draw[thick,dashed,domain=-2:0] plot(1.3,\x);
 \draw[thick,dashed,domain=0:1.7] plot(1.3+\x,\x);
 \draw[thick,domain=0:2] plot(\x+1,\x+2);
 \draw[thick,domain=-3:0] plot(\x+1,2);
 \draw[thick,domain=-4:0] plot(1,\x+2);
 \draw[thick,domain=0:1] plot(\x+2,\x+4);
 \draw[thick,domain=-4:0] plot(\x+2,4);
 \draw[thick,domain=-6:0] plot(2,\x+4);
 \draw[thick,domain=0:3.5] plot(\x-0.5,\x-1);
 \draw[thick,domain=-1.5:0] plot(\x-0.5,-1);
 \draw[thick,domain=-1:0] plot(-0.5,\x-1);
 \coordinate (A) at (0,-2);
 \draw (A) node[below]{$i$};
 \coordinate (B) at (1,-2);
 \draw (B) node[below]{$j$};
 \coordinate(C) at (1.4,-2);
 \draw (C) node[below]{$i'$};

\end{tikzpicture}
\caption{Six colored regions in the case where $j$ is higher than $i$}\label{Trop1}
\end{figure}
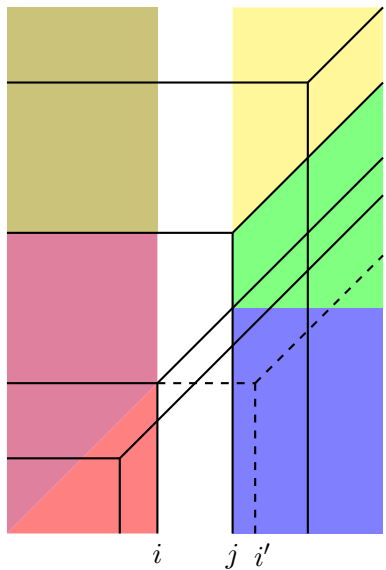
Next, assume the $y$-coordinate of the central point of $i$ is larger than that of $j$. 

 \begin{itemize}
\item The red region is the area bounded by the anti-diagonal ray and $y$-ray extending from the intersection of the $y$-ray of $i$ and the $x$-ray of $j$.
\item The purple region is bounded by the anti-diagonal ray and anti-$y$-ray extending from the intersection of the $y$-ray of $i$ and the $x$-ray of $j$, and $x$-ray of $i$.
\item The olive region is bounded by the anti-$y$ ray and $x$-ray of $i$.
\item The blue region is bounded by the anti-$x$-ray and $y$-ray of $j$.
\item The green region is bounded by the anti-$x$-ray, anti-$y$-ray of $j$, and the diagonal ray of $i$. 
\item The yellow region is bounded by the anti-$y$-ray of $j$ and diagonal ray of $i$.
\end{itemize}
\begin{figure}[h]
\begin{tikzpicture}
 \fill[red!50](-2,-2)--(0,0)--(0,-2)--(-2,-2);
 \fill[blue!50](1,-2)--(1,0)--(3,0)--(3,-2);
 \fill[green!50](3,0)--(1,0)--(1,3)--(3,5);
 \fill[yellow!50](1,3)--(1,5)--(3,5);
 \fill[purple!50](-2,-2)--(0,0)--(0,2)--(-2,2);
 \fill[olive!50](-2,2)--(0,2)--(0,5)--(-2,5);
 \draw[thick,domain=0:3] plot(\x,\x+2);
 \draw[thick,domain=-2:0] plot(\x,2);
 \draw[thick,domain=-4:0] plot(0,\x+2);
 \draw[thick,dashed,domain=0:1.3] plot(\x,2);
 \draw[thick,dashed,domain=-4:0] plot(1.3,\x+2);
 \draw[thick,dashed,domain=0:1.7] plot(1.3+\x,\x+2);
 \draw[thick,domain=0:2] plot(\x+1,\x);
 \draw[thick,domain=-3:0] plot(\x+1,0);
 \draw[thick,domain=-2:0] plot(1,\x);
 \draw[thick,domain=0:0.5] plot(\x+2,\x+4.5);
 \draw[thick,domain=-4:0] plot(\x+2,4.5);
 \draw[thick,domain=-6.5:0] plot(2,\x+4.5);
 \draw[thick,domain=0:3.5] plot(\x-0.5,\x-1);
 \draw[thick,domain=-1.5:0] plot(\x-0.5,-1);
 \draw[thick,domain=-1:0] plot(-0.5,\x-1);
 \coordinate (A) at (0,-2);
 \draw (A) node[below]{$i$};
 \coordinate (B) at (1,-2);
 \draw (B) node[below]{$j$};
 \coordinate(C) at (1.4,-2);
 \draw (C) node[below]{$i'$};

\end{tikzpicture}
\caption{Six colored regions in the case where $i$ is higher than $j$}\label{Trop1'}
\end{figure}
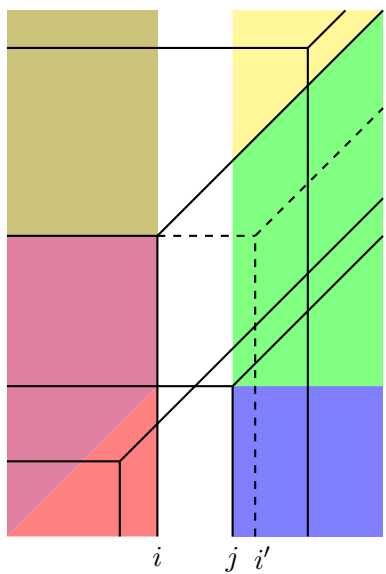
We define the condition $(\boldsymbol{\ast})$ as follows.
\begin{align*}
(\boldsymbol{\ast})
\begin{cases}
\text{The red region is not empty.}\cdots(a)\\
\text{The blue and olive regions are empty.}\cdots(b)\\
\text{Either the yellow or green region is non-empty.}\cdots(c)\\
\text{There are at least two elements in the red and purple regions.}\cdots(d)\\
\end{cases}
\end{align*}


Note that if the condition (a) in ($\boldsymbol{\ast}$) is false, then nothing happens after we exchange $i$ and $j$, so it is natural to assume (a).

Our main result uses the following $w$ and $F$.
\medskip
\begin{defi}\label{wFdef}
Let $w_{(i,j)}$ be the $3\times n$ matrix such that all the columns are zero vectors except for the columns corresponding to $i$ and $j$, the column corresponding to $i$ is $(1,-1,0)^T$, and the column corresponding to $j$ is $(-1,1,0)^T$. Namely,
\begin{equation*}
w_{(i,j)}:=
\begin{blockarray}{ccccccccccc}
\ &\ &\ & i &\ &\ &\ & j&\ &\  \  \\
\begin{block}{(ccccccccccc)}
0& \cdots& 0& 1& 0& \cdots& 0& -1& 0& \cdots 0\\
0& \cdots& 0& -1& 0& \cdots& 0& 1& 0& \cdots 0\\  
0& \cdots& 0& 0& 0& \cdots& 0& 0& 0& \cdots 0\\
\end{block}
\end{blockarray}
\in \mathbb{R}^{3\times n}.
\end{equation*}
The lattice polytope $F_{(i,j)}$ in $\mathbb{R}^{3\times n}$ is defined by the convex hull of $0$ and $f$, where $0$ denotes the zero matrix and $f$ is the matrix such that 
the columns corresponding to green or yellow regions are $(1,-1,0)^T$, and the columns corresponding to $i$ or $j$  are $(0,-1,0)^T$
 and the remaining columns are all zero vectors. Namely,
\begin{align*}
F_{(i,j)}&:=\mathrm{Conv}\left\{0,\ f
\right\} \subset \mathbb{R}^{3\times n}\\
\text{where} \ f&:=
\begin{blockarray}{ccccccccccc}
\ &\textcircled{1} \ &\ &i &j &\ &\textcircled{2} &\ &\ &\textcircled{3} \\
\begin{block}{(ccc|cc|ccc|ccc)}
0& \cdots& 0& 0& 0& 1& \cdots& 1& 0& \cdots& 0\\
0& \cdots& 0& -1& -1& -1& \cdots& -1& 0& \cdots& 0\\  
0& \cdots& 0& 0& 0& 0& \cdots& 0& 0& \cdots& 0\\
\end{block}
\end{blockarray}
\in \mathbb{R}^{3\times n}.
\end{align*}
Remark that $F_{(i,j)} \subset w_{(i,j)}^{\perp}$. 
We associate the columns corresponding to the tropical line that belongs to the red region with \textcircled{1}. Similarly, we associate the columns corresponding to the green or yellow region with \textcircled{2}, the purple region with \textcircled{3}.
\end{defi}
\begin{thm}\label{TCMmain}
Assume that there are two matching fields $\Lambda ,\ \Lambda '$, both of which have the same tropical hyperplane arrangements except for swapping two adjacent tropical lines, $i$ and $j$, and that condition $(\boldsymbol{\ast})$ is satisfied. Then, the pair of matching field polytopes $P_\Lambda ,\ P_{\Lambda '}$ can be obtained from one another by a combinatorial mutation.
\end{thm}
\begin{proof}
We use $w_{(i,j)}$ and $F_{(i,j)}$ as in Definition \ref{wFdef}, and let $w$ and $F$ be $w_{(i,j)}$ and $F_{(i,j)}$, respectively. Note that $u_{\mathrm{min}} = -1$ is only true for vertices corresponding to $\{i,j,k\}$ with $k \in \textcircled{1}$. Clearly, the difference between the two matching fields is only about $\{i,j,k\}$ ($k \in \textcircled{1}$). Therefore, take $\varphi_{w,F}$ as a candidate for the tropical map that induces a combinatorial mutation. 
Note that $\varphi_{w,F}$ sends the vertices of $P_\Lambda$ to the vertices of $P_\Lambda '$.

By these discussions, it is sufficient to prove that $\varphi_{w,F}(P_\Lambda)$ does not have new vertices. To show this
, we use Lemma \ref{lem:no_edge_in_image}. We only need to focus on the vertices of the matching field polytope $P_\Lambda$ which  
has a non-zero inner product $\lambda_{i,j}$ with $w$.
$\lambda_{i,j}$ is 1 only in the following cases because of the entries of $f$.

$$\begin{cases}
\begin{pmatrix}
k_1\\
k_2\\
k_3
\end{pmatrix}
(k_1 \in \textcircled{2}, k_2,k_3 \in \textcircled{1},\textcircled{3})\ (\text{\textbf{Case \ref{k123}}})\\
\begin{pmatrix}
k_1\\
k_2\\
i
\end{pmatrix}
(k_1 \in \textcircled{2},k_2 \in \textcircled{3})\ (\text{\textbf{Case \ref{k12i}}})\\
\begin{pmatrix}
k_1\\
k_2\\
j
\end{pmatrix}
(k_1 \in \textcircled{2},k_2 \in \textcircled{3})\ (\text{\textbf{Case \ref{k12j}}})
\end{cases}$$
The conditions (c) and (d) in ($\boldsymbol{\ast}$) assure the existence of those that satisfy with $\lambda_{i,j} = 1$. Thanks to the condition (b) in ($\boldsymbol{\ast}$), we only check three patterns. We now proceed by taking the cases on the pattern. 
\begin{case}\label{k123}
In this case, we swap the entries in the first row. The number on the top of the vector $v$ is the inner product $\langle f,v\rangle$.

\begin{tabular}{|l|r|} 
\multicolumn{2}{c}{$1$\ \ \ $-1$ } \\ \hline 
$k_1$& $j$\\ 
$k_2$& $i$\\
$k_3$& $k$\\ \hline
\end{tabular}
=
\begin{tabular}{|l|r|} 
\multicolumn{2}{c}{$0$\ \ \ $0$ } \\ \hline 
$j$& $k_1$\\ 
$k_2$& $i$\\
$k_3$& $k$\\ \hline
\end{tabular}
\end{case}
\begin{case}\label{k12i}

If the $y$-coordinate of the central point of $i$ is larger than that of $j$
, then $k_2$ belongs to the olive region. Therefore nothing is to prove (see the condition (b) in ($\boldsymbol{\ast}$)). Consider the case where 
the $y$-coordinate of the central point of $j$ is larger than that of $i$.
In this case, we swap the entries in the first row.

\begin{tabular}{|l|r|} 
\multicolumn{2}{c}{$1$\ \ \ $-1$ } \\ \hline 
$k_1$& $j$\\ 
$k_2$& $i$\\
$i$& $k$\\ \hline
\end{tabular}
=
\begin{tabular}{|l|r|} 
\multicolumn{2}{c}{$0$\ \ \ $0$ } \\ \hline 
$j$& $k_1$\\ 
$k_2$& $i$\\
$i$& $k$\\ \hline
\end{tabular}
\end{case}
\begin{case}\label{k12j}
If the $y$-coordinate of the central point of $i$ is larger than that of $j$
, then $k_2$ belongs to the olive region. Therefore nothing is to prove (see the condition (b) in ($\boldsymbol{\ast}$)). Consider the case where 
the $y$-coordinate of the central point of $j$ is larger than that of $i$.
In this case, we swap the entries in the second row.

\begin{tabular}{|l|r|} 
\multicolumn{2}{c}{$1$\ \ \ $-1$ } \\ \hline 
$k_1$& $j$\\ 
$k_2$& $i$\\
$j$& $k$\\ \hline
\end{tabular}
=
\begin{tabular}{|l|r|} 
\multicolumn{2}{c}{$0$\ \ \ $0$ } \\ \hline 
$k_1$& $j$\\ 
$i$& $k_2$\\
$j$& $k$\\ \hline
\end{tabular}
\end{case}
This completes the proof.
\end{proof}
\medskip
\begin{rem}
Even if we drop (c) and (d) in ($\boldsymbol{\ast}$), $\varphi_{w,F}$ becomes a unimodular transformation, so we can discuss a sequence of combinatorial mutations.
\end{rem}
\begin{exa}
We check that there is a sequence of combinatorial mutations from the block diagonal matching field $\mathcal{B}_2$ to the diagonal matching field. We want to move the red line to the dashed line of Figure \ref{BDMFtoDMF}.
See \cite[Figure 2]{mohammadi2019toric} about tropical hyperplane arrangements of block diagonal matching fields.
First, remark the condition (a) in ($\boldsymbol{\ast}$) is false on the left figure, then nothing happens after we exchange $i$ and $j$ in each step.
Therefore, nothing is to be checked.

Next, see the second figure. Check the condition $(\boldsymbol{\ast})$ for the red line and the left adjacent line. Then, the two lines are swapped. Repeating these operations for $4$ times, we get the right figure.  For exapmle, the second case of the sequence of the combinatorial mutation, $w_{(i,,j)}$ and $F_{(i,,j)}$ are the following:

\begin{equation*}
w_{(i,j)}:=
\begin{blockarray}{cccccc}
\ &\ &\ &i &j &\ \\
\begin{block}{(cccccc)}
0& 0& 0& 1& -1& 0\\
0& 0& 0& -1& 1& 0\\  
0& 0& 0& 0& 0& 0\\
\end{block}
\end{blockarray}
\quad \text{  and } \quad
F_{(i,j)}:=\mathrm{Conv}\left\{0,\ 
\begin{blockarray}{cccccc}
\textcircled{1} &\ &\textcircled{3} &i &j &\textcircled{2} \\
\begin{block}{(c|cc|cc|c)}
0& 0& 0& 0& 0& 1\\
0& 0& 0& -1& -1& -1\\  
0& 0& 0& 0& 0& 0\\
\end{block}
\end{blockarray}
\right\}
.
\end{equation*}
The number in the column of the matices corresponds to the number counting from the right in the second figure.
\end{exa}
\begin{rem}\label{BDMFCor}
By repeating the above operation,  any block diagonal matching field polytopes can be obtained from the diagonal matching field polytope by a sequence of combinatorial mutations. This is one of the main results of \cite{clarke2021combinatorial}.
\end{rem}
\medskip
\begin{figure}
\centering
\begin{minipage}[h]{0.3\linewidth}
\begin{tikzpicture}[xscale=0.4,yscale=0.1]
\wTrop{3}{0}
\wTrop{1}{12.2}
\wTrop{0.5}{6.1}
\wTrop{1.5}{18.3}
\wTrop[red]{2.5}{-6.1}
\wTrop{2}{24.4}
\draw[thick,dashed,domain=-13.9:0] plot(-0.5,-6.1+\x);
\draw[thick,dashed,domain=7.5:0] plot(-0.5+\x,-6.1+\x);
\draw[->,very thick] (9,5)--(11.5,5);
\end{tikzpicture}
\end{minipage}
\begin{minipage}[h]{0.03\linewidth}
\centering
\begin{tikzpicture}[scale=0.6]
\end{tikzpicture}
\end{minipage}
\begin{minipage}[h]{0.3\linewidth}
\begin{tikzpicture}[xscale=0.4,yscale=0.1]
\wTrop[red]{3}{0}
\wTrop{1}{12.2}
\wTrop{0.5}{6.1}
\wTrop{1.5}{18.3}
\wTrop{-0.5}{-6.1}
\wTrop{2}{24.4}
\draw[thick,dashed,domain=-20:0] plot(0,\x);
\draw[thick,dashed,domain=7:0] plot(\x,\x);
\draw[->,very thick] (9,5)--(11.5,5);
\end{tikzpicture}
\end{minipage}
\begin{minipage}[h]{0.03\linewidth}
\begin{tikzpicture}[scale=0.6]
\end{tikzpicture}
\end{minipage}
\begin{minipage}[h]{0.3\linewidth}
\begin{tikzpicture}[xscale=0.4,yscale=0.1]
\wTrop{0}{0}
\wTrop{1}{12.2}
\wTrop{0.5}{6.1}
\wTrop{1.5}{18.3}
\wTrop{-0.5}{-6.1}
\wTrop{2}{24.4}
\end{tikzpicture}
\end{minipage}
\caption{The example of combinatorial mutations from the block diagonal matching field $\mathcal{B}_2$ to the diagonal matching field. (The $y$-axis is scaled down by $\frac{1}{4}$.)}\label{BDMFtoDMF}
\end{figure}
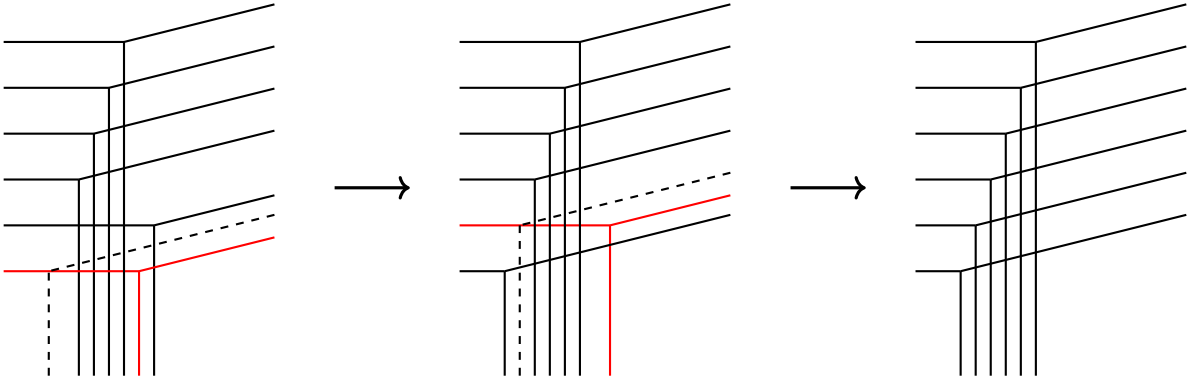

\bibliography{Biblio}
\bibliographystyle{abbrv}
\end{document}